\documentclass{amsart}
\usepackage{amsfonts,latexsym,rawfonts,amsmath,amssymb, amsthm}
\usepackage{latexsym,lscape,rawfonts}
\usepackage{mathrsfs, enumitem}
\usepackage{galois}
\usepackage[all,cmtip]{xy}
\usepackage{extarrows,color}

\usepackage{times}

\newenvironment{thmbis}[1]
 {\renewcommand{\thethm}{\ref{#1}$ $}%
 \addtocounter{thm}{-1}%
 \begin{thm}}
 {\end{thm}}

\newtheorem{thm}{Theorem}[section]

\newtheorem*{fixed point criterion}{Fixed point criterion}

\newtheorem{lem}[thm]{Lemma}
\newtheorem{prop}[thm]{Proposition}

\theoremstyle{definition}
\newtheorem{defn}[thm]{Definition}
\newtheorem{exam}[thm]{Example}
\newtheorem{ques}[thm]{Question}
\theoremstyle{remark}
\newtheorem{rem}[thm]{Remark}
\numberwithin{equation}{section}

\newcommand{\Z}{\mathbb Z}
\newcommand{\Q}{\mathbb Q}

\newcommand{\fix}{\mathrm{Fix}}

\newcommand{\ind}{\mathrm{ind}}
\newcommand{\chr}{\mathrm{chr}}
\newcommand{\ichr}{\mathrm{ichr}}
\newcommand{\rk}{\mathrm{rk}}

\newcommand{\aut}{\mathrm{Aut}}
\newcommand{\out}{\mathrm{Out}}
\newcommand{\Out}{\mathrm{Out}}

\newcommand{\inn}{\mathrm{Inn}}

\newcommand{\stab}{\mathrm{Stab}}

\newcommand{\trace}{\mathrm{Trace}}

\newcommand{\F}{\mathbf{F}}      

\begin{document}
	\title{Iwip endomorphisms of free groups and fixed points of graph selfmaps}
	\author{Peng Wang}
	\address{School of Mathematics and Statistics, Xi'an Jiaotong University, Xi'an 710049, CHINA}
	
	\email{wp580829@stu.xjtu.edu.cn}
	
	\author{Qiang Zhang}
\address{School of Mathematics and Statistics, Xi'an Jiaotong University, Xi'an 710049, CHINA}

\email{zhangq.math@mail.xjtu.edu.cn}

\date{\today}
	
\thanks{The authors are partially supported by NSFC (No. 12271385 and 12471066), the Shaanxi Fundamental Science Research Project for Mathematics and Physics (No. 23JSY027), and the Fundamental Research Funds for the Central Universities.}


	\begin{abstract}
In a paper from 2011, Jiang, Wang and Zhang studied the fixed points and fixed subgroups of selfmaps on a connected finite graph or a connected compact hyperbolic surface $X$. In particular, for any selfmap $f: X\to X$, they proved that a certain quantity defined in terms of the characteristic $\chr(f, \F)$ and the index $\ind(f, \F)$ of a fixed point class $\F$ of $f$ is bounded below by $2\chi(X)$, where $\chi(X)$ is the Euler characteristic of $X$. In this paper, we give a sufficient condition for when equality holds and hence we partially answer a question of Jiang. We do this by studying iwip outer endomorphisms of free groups acting on stable trees.
	\end{abstract}

\maketitle


\section{Introduction}\label{Sect. introduction}

For a selfmap $f$ of a compact connected polyhedron $X$, Nielsen fixed point theory is concerned with the properties of the fixed point set
$$\fix f:=\{x\in X\mid f(x)=x\},$$
which are invariant under any homotopy of the selfmap $f$ (see \cite{fp1, JZ} for an introduction to Nielsen fixed point theory).

The fixed point set $\fix f$ splits into a disjoint union of \emph{fixed point classes}: two fixed
points $x$ and $x'$ are in the same class if and only if there is a lifting $\tilde f: \widetilde X\to \widetilde X$ of $f$ such that $x, x'\in p\fix( \tilde f)$, where $p:\widetilde X\to X$ is the universal cover. For each fixed point class $\F$ of $f$, a classical homotopy invariant \emph{index} $\ind(f,\F)\in \Z$ that plays a central role in Nielsen fixed point theory (e.g. \cite{Gr, GJ, J98, ZZ3}), is well-defined by using homology. The famous Lefschetz-Hopf fixed point theorem says that the sum of the indices of the fixed points of $f$ is equal to the \emph{Lefschetz number} $L(f)$, which is defined as
$$L(f):=\sum_q(-1)^q\trace (f_*: H_q(X;\Q)\to H_q(X;\Q)).$$

A fixed point class is \emph{essential} if its index is non-zero. The number of essential fixed point classes of $f$ is called the \emph{Nielsen number} of $f$, denoted by $N(f)$, which is a lower bound for the minimum number $M(f)$ of fixed points in the homotopy class of $f$. In \cite{J80}, Jiang showed that $M(f)=N(f)$ for every selfmap $f: X\to X$ if $X$ has no local separating points and $X$ is not a compact surface.

In geometric group theory, there is an analogue of fixed point set in Nielsen theory: for an endomorphism $\phi$ of a group $G$, its \emph{fixed subgroup} refers to the subgroup
$$\fix\phi:=\{g\in G\mid \phi(g)=g\},$$
which has many special and interesting properties. In the past fifty years, the studies on fixed subgroups of various groups were very active (see \cite{Ve02} for a survey and \cite{AJZ22, LZ23} for some new progress).

For a group $G$, let $\rk G$ denote the \emph{rank} of $G$, i.e., the minimal number of generators, and let $\aut(G)$ (resp. $\out(G)$) denote the group of automorphisms (resp. outer automorphisms) of $G$. For free groups, a celebrated result was due to Bestvina and Handel \cite{BH}. They introduced train track maps for automorphisms of free groups and then proved the famous Scott's conjecture: for any automorphism $\phi$ of a free group $F_n$ of rank $n$, $\rk\fix\phi \leq \rk F_n.$

In the paper \cite{JWZ}, Jiang, Wang and Zhang defined a new homotopy invariant \emph{characteristic} $\chr(f,\F)\in \mathbb{Z}$ for a fixed point class $\F$ of a selfmap $f$ of a polyhedron $X$ as follows:  for $x\in \F$,  let $f_{\pi}: \pi_1(X,x)\rightarrow \pi_1(X,x)$ be the induced endomorphism of the fundamental group.
The \emph{rank} of $\F$ is defined as $\rk(f,\F):=\rk\fix f_\pi,$   and
$$\chr(f,\F):=1-\rk(f,\F),$$
with the only exception that $\chr(f,\F):=\chi(X)=2-\rk(f,\F)$ when $X$ is a closed (orientable or not) surface and $\fix f_\pi=\pi_1(X)$. Moreover, the following result was proved, which relates the concepts of characteristics and indices of fixed point classes.

\begin{thm}[Jiang-Wang-Zhang, \cite{JWZ}]\label{JWZ main theorem}
Suppose $X$ is either a connected finite graph or a connected
compact hyperbolic surface, and $f: X\rightarrow X$ is a selfmap.
Then

$\mathrm{(A)}$ $\ind(f, \F)\leq \chr(f, \F)$ for every fixed point class
$\F$ of $f$;

$\mathrm{(B)}$ when $X$ is not a tree,
$$\sum_{\ind(f, \F)+\chr(f, \F)<0}\{\ind(f, \F)+\chr(f, \F)\}\geq 2\chi(X),$$
where the sum is taken over all fixed point classes $\F$ with
$\ind(f, \F)+\chr(f, \F)<0$.
\end{thm}

Note that in the above theorem, a surface $X$ is \emph{hyperbolic} means that $X$ has negative Euler characteristic. An analogue of Theorem \ref{JWZ main theorem} for $3$-manifolds can be found in \cite{Z12}. Furthermore, in the recent papers \cite{ZZ1, ZZ2}, Zhang and Zhao defined the \emph{improved characteristic} $\ichr(f,\F)$ (see Section \ref{sect 2} for a definition) for a $\pi_1$-injective selfmap  $f$ of a connected finite graph $X$, and improved the inequality (A) in Theorem \ref{JWZ main theorem} into an equality, which gives a group-theoretical approach for computing indices of fixed point classes of $\pi_1$-injective selfmaps of graphs.

\begin{thm}[Zhang-Zhao, \cite{ZZ2}]\label{ZZ2 main thm 1}
	Let $X$ be a connected finite graph and $f: X\to X$ a $\pi_1$-injective selfmap. Then
	$$\ind(f, \F)=\ichr(f, \F),$$
	for every fixed point class $\F$ of $f$.
\end{thm}

Now we consider the inequality (B) in Theorem \ref{JWZ main theorem}.

For a surface automorphism $f$, the formula (B) becomes an
equality if and only if in the Thurston canonical form of $f$, each
periodic piece is of period $1$, and each pseudo-Anosov piece
keeps all prongs into themselves. Thus every surface
automorphism has an iterate for which (B) is an equality. 

The situation of graph selfmaps is much more complicated. Note that the left hand side of (B) is a mutation (see \cite{J98} for a definition) invariant while the right hand side is not. Hence if a graph map $f$ is a mutation from a graph of lower rank, then (B) can never be an equality. For instance when $f$ is not $\pi_1$-injective, i.e., the induced endomorphism $f_{\pi}$ of the fundamental group is not injective. Therefore, Jiang \cite{J09} posed the following question:

\begin{ques}[Jiang, \cite{J09}]\label{Jiang's question}
For any $\pi_1$-injective graph selfmap $f: X\to X$, when does there exist some iterate $f^k$ such that (B) is an equality?
\end{ques}

In this paper, we give a partial answer of Question \ref{Jiang's question}.

\begin{thm}\label{main thm3}
Suppose $X$ is a connected finite graph but not a tree, and $f: X\rightarrow X$ is a selfmap. Let $f_\pi$ induce an iwip outer endomorphism $\Phi$ of the fundamental group $\pi_1(X)=F_n$. Then there is an integer $k\geq 1$ such that
$$\sum_{\ind(f^k, \F)+\chr(f^k, \F)<0} \{\ind(f^k, \F)+\chr(f^k, \F)\}= 2\chi(X),$$
if and only if
\begin{enumerate}
	\item $\Phi\in \out(F_n)$ and the stable tree $T_\Phi$ is geometric; or
	\item $\Phi\notin \out(F_n)$ and every branch point in $G_\Phi$ is $h$-periodic.
\end{enumerate}
\end{thm}

For detailed definitions of $T_\Phi$, $G_\Phi$ and the map $h$, see Section \ref{sect. iwip}. An outer endomorphism $\Psi$ of a free group $F_n$ is \emph{fully irreducible} (abbreviated as \emph{iwip}) if no positive power $\psi^k$ ($\psi\in \Psi$) maps a proper nontrivial free factor into a conjugate of itself. Being an iwip is one of the analogs for free groups of being pseudo-Anosov for mapping classes of hyperbolic surfaces. Roughly speaking, \emph{geometric} means that the tree $T_\Phi$ is transverse to a measured foliation on a finite $CW$-complex. For example, if $\Phi$ comes from a pseudo-Anosov mapping class on a compact surface $S$ with one boundary component, then $\Phi$ is iwip and $T_\Phi$ is geometric because $T_\Phi$ lives in the Thurston boundary of Teichm{\"u}ller space: it is dual to a measured foliation on $S$. See Section \ref{sect. iwip} for more information.

To prove Theorem \ref{main thm3}, we investigate iwip endomorphisms of a free group $F_n$, and show a result as follows (for an analog of automorphisms, see Coulbois and Hilion \cite{CH1}).

\begin{thm}\label{main thm2}
For every iwip outer endomorphism $\Phi$ of a free group $F_n$ of rank $n$, let $T_\Phi$ be the stable tree of $\Phi$. Then for all $m>0$, we have
$$
2\ind(\Phi^m)\leq\ind_{\mathrm{geo}}(T_\Phi)
\left\{
\begin{array}{ll}
  \leq 2n-2, &  \Phi\in \out(F_n)\\
  =2n-2, & \Phi\notin \out(F_n)
\end{array}
\right..
$$
Moreover, there exists $k>0$ such that
$$ 2\ind(\Phi^k)=\ind_{\mathrm{geo}}(T_\Phi)=2n-2,$$
if and only if
\begin{enumerate}
	\item $\Phi\in \out(F_n)$ and $T_\Phi$ is geometric; or
	\item $\Phi\notin \out(F_n)$ and every branch point in $G_\Phi$ is $h$-periodic, where $G_\Phi$ is the quotient graph of $T_\Phi$ by the $F_n$-action, and $h:G_\Phi\rightarrow G_\Phi$ is the map induced by the homothety $H$ associated to an iwip endomorphism $\phi\in\Phi$.
\end{enumerate}
\end{thm}

This paper is organized as follows. In Section \ref{sect 2}, we introduce the concept $\ind(\Psi)$ for any outer endomorphism $\Psi$ of a free group. In Section \ref{sect 3}, we study the relationship between  $\ind(\Psi)$ and $\ind(f, \F)$ of a fixed point class $\F$ of a graph selfmap $f$. In Section \ref{sect. iwip}, we give some facts and results on the geometric index $\ind_{\mathrm{geo}}(T_\Psi)$ of iwip endomorphisms. Finally in Section \ref{sect mani results} and Section \ref{sect last}, we prove the main theorems of this paper and give some examples.


\section{Index of an outer endomorphism}\label{sect 2}

In this section, we review some definitions and results on endomorphisms of a free group $F_n$ of rank $n$ (see \cite[Sect. 2.1]{ZZ1} for more details).

\subsection{Attracting fixed point}\label{subsect. Attracting fixed point}

An injective endomorphism $\phi$ of the free group $F_n$ extends to a continuous
injective map $\partial\phi: \partial F_n\to \partial F_n$ of the boundary at infinity $\partial F_n$. The boundary $\partial \fix\phi$ naturally embeds
in $\partial F_n$. A fixed point $W$ of $\partial\phi$ is \emph{attracting} if $W\not\in \partial \fix\phi$ and if there exists an element $u\in F_n$ such that $\phi^k(u)$ converges to $W$. Let $\mathcal A(\phi)$ be the set of attracting fixed points of $\phi$. Two attracting fixed points $W, V\in \mathcal A(\phi)$ are called \emph{equivalent} if
$W=gV$ for some $g\in\fix\phi$. Let
$$a(\phi):=\#(\mathcal A(\phi)/\fix\phi),$$
the cardinality of $\mathcal A(\phi)/\fix\phi$, or equivalently, the number of orbits of $\fix\phi$ acting on $\mathcal A(\phi)$.

\subsection{Isogrediency class}
For an endomorphism $\phi$ of $F_n$, the \emph{outer endomorphism} of $\phi$ is
$$\Phi:=\{i_m\comp\phi\mid m\in F_n\}.$$
So two endomorphisms $\phi_1,\phi_2$ of $F_n$ represent the same outer endomorphism $\Phi$ if there exists $m\in F_n$ such that $\phi_2=i_m\comp\phi_1$ with $i_m(g)=mgm^{-1}$.
Moreover, $\phi_1$ and $\phi_2$ are said to be \emph{isogredient} (or \emph{similar} in \cite{GJLL}) if $m$ can be written $m=c\phi_1 (c^{-1})$ for some $c\in F_n$, equivalently, if $\phi_2=i_c\comp\phi_1\comp (i_c)^{-1}$. Isogredient endomorphisms represent the same endomorphism, up to a change of basis in $F_n$. In particular, the rank of $\fix\phi$ and the number $a(\phi)$ are isogredient invariants.
\begin{rem}
In \cite{ZZ1}, an outer endomorphism $\Phi$ contained $\phi$ is called an \emph{Inn-coset} and written as $\inn\phi$.
\end{rem}

\subsection{Index of an outer endomorphism}

In \cite{GJLL}, Gaboriau, Jaeger, Levitt and Lustig studied the fixed subgroups of automorphisms of free groups and showed

\begin{thm}[Gaboriau et al. \cite{GJLL}]\label{GJLL thm}
Let $\phi_0,\cdots,\phi_k$ be automorphisms of $F_n$ represent the same outer automorphism $\Phi\in \out(F_n)$ and belong to distinct isogrediency classes. Then
$$\sum^{k}_{i=0}\{\rk\fix\phi_i+\frac{1}{2}a(\phi_i)-1\}\leq n-1.$$
Equivalently, $\sum_{\phi\in\mathscr{S}(\Phi)}\max\{0, ~\rk\fix\phi+\frac{a(\phi)}{2}-1\}\leq n-1$, where $\mathscr{S}(\Phi)$ is the set of isogrediency classes of automorphisms representing a given outer automorphism $\Phi\in \Out(F_n)$.
\end{thm}

Following Theorem \ref{GJLL thm}, Gaboriau et al. gave a concept of $index$ of an automorphism $\phi\in \aut(F_n)$ as
$$\ind(\phi):=\rk\fix\phi+\frac{1}{2}a(\phi)-1, $$
and the $index$ of an outer automorphism $\Phi\in \out(F_n)$ as
\begin{equation}\label{def ind for Out}
\ind(\Phi):=\sum_{\phi\in\mathscr{S}(\Phi)}\max\{0,~\ind(\phi)\}.
\end{equation}
Therefore, Theorem \ref{GJLL thm} is equivalent to the inequality
$\ind(\Phi)\leq n-1,~~~~~~\forall\Phi\in \Out(F_n).$

Recently, Zhang and Zhao \cite[Theorem 2.6]{ZZ1} extended the above result (Theorem \ref{GJLL thm}) of automorphisms to that of injective endomorphisms. So, we define the \emph{index for outer injective endomorphisms} of free groups as in equation (\ref{def ind for Out}), and have an analogous inequality as follows.

\begin{thm}\label{ZZ1 thm2.6'}
Let $\Phi$ be an outer injective endomorphism of a free group $F_n$. Then
\begin{equation}\label{2' equ}
\ind(\Phi)\leq n-1.
\end{equation}
\end{thm}

\begin{proof}
The result is equivalent to \cite[Theorem 2.6]{ZZ1} clearly.
\end{proof}

\section{Relationship between Theorem \ref{JWZ main theorem} and Theorem \ref{ZZ1 thm2.6'}}\label{sect 3}

In this section, we first review some definitions and results on fixed point classes, see \cite[Sect. 3]{ZZ1} for more details.

Let $f:X\rightarrow X$ be a $\pi_1$-injective selfmap of a connected finite graph $X$, $p: \tilde{X}\to X$ a universal cover of $X$, with group $\pi$ of covering transformations identified with the fundamental group $\pi_1(X)=F_n$.

For any lifting $\tilde{f}:\tilde{X}\to \tilde{X}$ of $f$, the projection of its fixed point set is called a \emph{fixed
point class} of $f$, written $\F=p(\fix\tilde{f})$. When $\fix\tilde{f}= \emptyset$, we call $p(\fix\tilde{f})$ an \emph{empty} fixed point class. Each lifting $\tilde{f}$ induces an endomorphism $\tilde{f}_{\pi}:\pi\to\pi$ defined by
$$\tilde{f}\comp \gamma=\tilde{f}_{\pi}(\gamma)\comp \tilde{f},\quad \gamma\in \pi.$$
If two liftings $\tilde{f}$ and $\tilde{f'}$ label the same fixed point class $\F=p(\fix\tilde f)=p(\fix\tilde f')$, then there exists $\gamma\in \pi$ such that $\tilde{f'}=\gamma\comp \tilde{f}\comp \gamma^{-1}$, and hence the induced injective endomorphisms are isogredient
$\tilde f'_{\pi}=i_{\gamma}\comp\tilde f_{\pi}\comp i_{\gamma^{-1}}$. (Note that empty fixed point classes have the same index $0$ but may have different labels and hence be regarded as different.) Therefore,
$$\rk(f,\F):=\rk\fix(\tilde f_{\pi})=\rk\fix (\tilde f'_{\pi}), \quad\quad a(f, \F):=a(\tilde f_{\pi})=a(\tilde f'_{\pi}).$$
The \emph{improved characteristic} is defined to be
\begin{equation}\label{equa. ichr psi}
\ichr(f,\F):=1-\rk(f,\F)-a(f,\F)=1-\rk\fix(\tilde f_{\pi})-a(\tilde f_{\pi}).
\end{equation}
Zhang and Zhao (Theorem \ref{ZZ2 main thm 1}) showed that for every $\pi_1$-injective selfmap $f$ of a connected finite graph $X$,
$$\ind(f,\F)= \ichr(f,\F).$$
Note that the index $\ind(f,\F)$ was initially defined by using homology, now the above equality gives a group-theoretical approach for computing indices of fixed point classes of $\pi_1$-injective selfmaps of graphs.
Moreover, we have the following.

\begin{thm}\label{2 thm}
Let $X$ be a connected finite graph but not a tree. Then for every $\pi_1$-injective selfmap $f$ of $X$, inequality (B) in Theorem \ref{JWZ main theorem} is equivalent to inequality (\ref{2' equ}) in Theorem \ref{ZZ1 thm2.6'}, in particular, the equal signs in them hold simultaneously.
\end{thm}

\begin{proof}
For an injective endomorphism $\phi$ of $F_n$, let $\Phi$ be the outer endomorphism represented by $\phi$. Then there exists a $\pi_1$-injective selfmap $f:X\rightarrow X$ of a connected finite graph $X$ with $\pi_1(X)=F_n$ representing $\Phi$, i.e., for any $i_c\comp\phi\in \Phi$, there is a lifting $\tilde f:\tilde X\to \tilde X$ such that $i_c\comp\phi=\tilde{f}_{\pi}$.

If $n=1$, it is trivial. Now we suppose $n\geq 2$. As in the proof of \cite[Theorem 2.6]{ZZ1}, there is a one-to-one correspondence $\F\mapsto [\psi]$ between the fixed point classes of $f$ and isogrediency classes contained in $\Phi$, with
\begin{equation}\label{equa. ichr psi}\nonumber
\ichr(f,\F)=1-\rk(f,\F)-a(f,\F)=1-\rk\fix\psi-a(\psi).
\end{equation}
By Theorem \ref{ZZ2 main thm 1}, for any fixed point class $\F$ of $f$, we have
$$\ind(f,\F)+\chr(f,\F)=-2(\rk\fix\psi+a(\psi)/2-1)=-2\ind(\psi).$$
Note that the Euler characteristic $\chi(X)=1-n$, so inequality (B) in Theorem \ref{JWZ main theorem} is equivalent to
$$\ind(\Phi)=\sum_{\psi\in\mathscr{S}(\Phi)}\max\{0, ~\ind(\psi)\}\leq n-1,$$
where the sum runs over all distinct isogrediency classes contained in $\Phi$.
\end{proof}

\begin{rem}
Theorem \ref{2 thm} gives the relationship between the Nielsen fixed point index and the index of endomorphisms, namely, we translate the topological problem to an algebraic one. In the following, we only need to consider the algebraic version.
\end{rem}


\section{Iwip endomorphism, stable tree and geometric index}\label{sect. iwip}

Recall that an outer endomorphism $\Phi$ of a free group $F_n$ is \emph{fully irreducible} (\emph{iwip}) if no positive power $\phi^k$ maps a proper nontrivial free factor into a conjugate of itself,  which is an analogue of a pseudo-Anosov homeomorphism of a closed surface. The notion of iwip automorphisms of free groups was introduced by Bestvina and Handel \cite{BH} inspired by Thurston's work on surface homeomorphisms. Reynolds \cite{R} extended the results in \cite{BH} to irreducible endomorphisms.
Note that, irreducible in \cite{R} means fully irreducible, i.e., iwip in the sense of many other papers, such as Coulbois-Hilion \cite{CH1} and this paper.

To introduce the properties of iwip endomorphisms and the geometric index $\ind_{\mathrm{geo}}(T)$, we introduce train track maps at first.
Let $\Phi$ be an outer endomorphism of $F_n$, and $\phi\in \Phi$ an endomorphism in the following.
\subsection{Train track map}
Let $R_n:=\lor^n_{i=1}S^1$ be the wedge of $n$ copies of the circle $S^1$, we make the identification $F_n=\pi_1(R_n, \ast)$ where $\ast$ is the unique vertex of $R_n$. A \emph{marked graph} is a finite graph $G$ with a homotopy equivalence $\tau: R_n\rightarrow G$. The marking $\tau$ induces an isomorphism $\tau_\pi: F_n=\pi_1(R_n, \ast)\to \pi_1(G, \tau(\ast))$. A map (resp. homotopy equivalence) $f: G\to G$ defines an outer endomorphism (resp. automorphism) $\Phi$ of $F_n$.

For a marked graph $G$, let $V$ denote the set of vertices of $G$, and let $E$ denote the set of edges of $G$. A map $f:G\rightarrow G$ is a \emph{topological representative} for an outer endomorphism $\Phi$ of $F_n$ if:

$\bullet$ $f(V)\subseteq V$,

$\bullet$ for any $e\in E$, $f|_e$ is either locally injective, or $f(e)$ is a vertex, and

$\bullet$ $f$ induces $\Phi$.

It is clear that a topological representative always exists. A topological representative $f$ is $admissible$ if there is no invariant non-degenerate forest. An admissible topological representative $f:G\rightarrow G$ for $\Phi$ is called a \emph{train track map} if
$$[f^r(e)]=f^r(e)$$
for every edge $e$ of $G$, where $[f^r(e)]$ is the immersed path homotopic to $f^r(e)$ rel. endpoints.
The \emph{transition matrix} of a topological representative $f$ is an $m\times m$ ($m$ is the number of edges in $G$) matrix $M(f)$ whose $(i,j)$-entry is the number of times the $f$-image of $e_j$ crosses $e_i$ (in either direction). Any transition matrix is a non-negative integral matrix, and it is evident that $M(f)^r=M(f^r)$ when $f$ is a train track map.

Note that iwip endomorphisms are injective (\cite[Corollary 3.3]{R}), and there always exists a train track map for every iwip outer endomorphism $\Phi$ of $F_n$, see \cite{DV}.

\subsection{Stable tree}

In this subsection, we introduce the stable tree $T_\Phi$ for an iwip outer endomorphism $\Phi$ of $F_n$ (for more details, see Reynolds \cite[Chapter 3]{R}).

Fix a train track representative $f:G\rightarrow G$ of $\Phi$. Let $m$ be the number of edges in $G$, $\lambda_\Phi>1$ be the Perron-Frobenius eigenvalue of the transition matrix $M$ of $f$,  and $\mathbf{v}=(v_1,\ldots, v_m)$ ($v_i>0$) the normalized eigenvector associated to $\lambda_\Phi$. Equip $G$ with the \emph{Perron-Frobenius metric}:  assign the edge $e_i$ with the length $v_i$, $i=1,\ldots, m$. With this metric, the map $f$ extends the length of legal paths by the factor $\lambda_\Phi$.

The universal cover $T_0:=\widetilde{G}$ is a simplicial tree (i.e., a 1-dimensional CW-complex containing no embedded
copy of a circle $S^1$) with lifted metric. The fundamental group $F_n$ acts by deck
transformations, and thus by isometries, on $T_0$. Let $\tilde{f}:T_0\rightarrow T_0$ be a lift of $f$. Then $\tilde{f}$ is corresponding to a unique endomorphism $\phi\in \Phi$ characterized by
$$\tilde{f}(gx)=\phi(g)\tilde{f}(x), \quad g\in F_n.$$
Let $T'_k$ denote the minimal $F_n$-invariant subtree of $T_0$ with the action twisted by $\phi^k$. Then $T'_k=\tilde{f}^k(T_0)$. Define $T_k$ to be $T'_k$ with the metric rescaled by $\lambda_\Phi^{-k}$. The sequence of $F_n$-trees $(T_k)$ converges in the Gromov-Hausdorff topology to the tree $T_\Phi$, called the \emph{stable tree} of $\Phi$. The map $\tilde{f}:T_0\rightarrow T_0$ induces a sequence of maps $f_k:T_k\rightarrow T_{k+1}$, which give rise to a map $$H:=f_{\phi}:T_{\Phi}\rightarrow T_{\Phi}$$ called a $homothety$ satisfying:

$\bullet$ $Length(H[x,y])=\lambda_\Phi \cdot Length([x,y])$,

$\bullet$ $H(gx)=\phi(g)H(x)$.

For an action of $F_n$ on a tree $T$, the \emph{stabilizer} $\stab(x)$ of a point $x\in T$ is the subgroup of $F_n$ consisting of elements fixing $x$, that is,
$$\stab(x):=\{g\in F_n \mid g(x)=x\}.$$
If the stabilizer of every point in $T$ is trivial, we say that the action is \emph{free}.

Note that the stable tree $T_{\Phi^k}= T_\Phi$ for any $k>0$. Moreover,

\begin{prop}\cite[Proposition 3.21]{R}\label{2 prop}
Let $\phi:F_n\rightarrow F_n$ be an iwip endomorphism contained in the outer endomorphism $\Phi$, and $\phi\notin \aut(F_n)$. Then the stable tree $T_{\Phi}$ is simplicial with a free $F_n$-action.
\end{prop}

For an iwip endomorphism $\phi\notin \aut(F_n)$, let $G_\Phi$ be the quotient graph of $T_\Phi$ by the $F_n$-action, and let 
$$h:G_\Phi\rightarrow G_\Phi$$ be the map induced by the homothety $H$ associated to $\phi\in\Phi$. We will use the map $h$ of $G_\Phi$ to depict the index $\ind(\Phi)$ (see Theorem \ref{main thm2}).

\begin{rem}
When $\Phi$ is an outer automorphism of $F_n$, the group $\Out(F_n)$ acts on the outer space $CV_n$ and its boundary $\partial CV_n$. Recall that the compactified outer space $\overline{CV_n}=CV_n\cup \partial CV_n$ is made up of (projective classes of) $\mathbb{R}$-trees with an action of $F_n$ by isometries which is minimal and very small. See \cite{V} for a survey on outer spaces.
In \cite{LL1}, Levitt and Lustig proved that an iwip outer automorphism $\Phi$ has North-South dynamics on $\partial CV_n$: it has a unique attracting fixed tree $T_{\Phi}$ and a unique repelling fixed tree $T_{\Phi^{-1}}$ in the boundary $\partial CV_n$:
$$T_{\Phi}\cdot \Phi=\lambda_{\Phi}T_{\Phi} \quad and \quad T_{\Phi^{-1}}\cdot \Phi=\frac{1}{\lambda_{\Phi^{-1}}}T_{\Phi^{-1}}$$
where $\lambda_{\Phi}>1$ is the \emph{expansion factor }of $\Phi$. 
\end{rem}

\subsection{Geometric index}
Gaboriau and Levitt \cite{GL} introduced an index for a tree $T$ in the compactification of the outer space $\overline{CV_n}$, we call it \emph{geometric index} denoted $\ind_{\mathrm{geo}}(T)$. It is defined by using the valence of the branch points in $T$, with an action of $F_n$ by isometries. Moreover, they proved
\begin{equation}\label{def of geo ind}
\ind_{\mathrm{geo}}(T)=\sum_{[x]\in T/F_n}\ind_{\mathrm{geo}}(x)\leq 2n-2,
\end{equation}
where the local index of a branch point $x$ in $T$ is
$$\ind_{\mathrm{geo}}(x)=\#(\pi_0(T\setminus \{x\})/\stab(x))+2\rk\stab(x)-2\geq 1.$$
Moreover, the equality $\ind_{\mathrm{geo}}(T)=2n-2$ holds if and only if the tree $T$ is geometric:

\begin{defn}
An $\mathbb{R}$-tree is called \emph{geometric} if it is transverse to a measured foliation on a finite CW-complex.
\end{defn}

\begin{rem}
For another way to define a geometric tree, see \cite{GL}.
\end{rem}

For example, we have

\begin{prop}\label{2n-2 for endo}
Let $T_\Phi$ be the stable tree of an iwip outer endomorphism $\Phi$ of $F_n$ and $\Phi\notin \out(F_n)$. Then $T_\Phi$ is geometric, and
$$\ind_{\mathrm{geo}}(T_\Phi)=2n-2.$$
Moreover, there are at most $2n-2$ orbits of branch points, and for every branch point $x\in T_\Phi$, we have $\#\pi_0(T_\Phi\setminus \{x\})\leq 2n$.
\end{prop}

\begin{proof}
Proposition \ref{2 prop} shows that the tree $T_{\Phi}$ is free, namely, the stabilizer of every point in $T_{\Phi}$ is trivial.
From \cite[Theorem III.2]{GL}, the geometric index of a tree equals to $2n-2$ if and only if the tree is geometric. Note that \cite[Example II.5]{GL} shows that a minimal simplicial $F_n$-action is geometric if
and only if every edge stabilizer has finite rank, so $T_{\Phi}$ is geometric, and hence $\ind_{\mathrm{geo}}(T_\Phi)=2n-2$. It follows
$$1\leq \ind_{\mathrm{geo}}(x)=\#\pi_0(T_\Phi\setminus \{x\})-2 \leq 2n-2$$
for every branch point $x\in T_\Phi$. So the assertion holds.
\end{proof}

\section{Proof of Theorem \ref{main thm2}}\label{sect mani results}

In this section, we focus on iwip endomorphisms and prove one of the main results of this paper.

Let $\phi: F_n\to F_n$ be an iwip endomorphism, and let $\Phi$ be the iwip outer endomorphism containing $\phi$.
Recall that the stable tree $T_\Phi$ in Proposition \ref{2 prop} has the same properties of the tree constructed in \cite[Theorem 2.1]{GJLL}. The homothety $H: T_\Phi\to T_\Phi$ associated to $\phi$ satisfies:
\begin{equation}\label{homothety equa.}
\phi(w)H=Hw,\quad w\in F_n,
\end{equation}
with the expansion factor $\lambda_{\Phi}>1$.  So the homothety $H$ has at most one fixed point, which is called the \emph{center} $C_H$ of $H$ if a fixed point exists. (Note that if $\phi\in\aut(F_n)$, the homothety $H$ with $\lambda_{\Phi}>1$ has a unique fixed point in the metric completion of $T_\Phi$, but we insist that the
fixed point is in $T_\Phi$. If $\phi\notin\aut(F_n)$, then $H$ is not invertible and should really be
called a ``homothetic embedding". Even though $T_\Phi$ is simplicial in this case, a center may fail to exist, see Example \ref{example referee} below.) As in \cite[Section 2]{GJLL}, we have the following observation:

\begin{lem}\label{property of Homothety}
If a homothety $H_{\phi}$ associates to an iwip endomorphism $\phi\in \Phi$ and $C_{H_\phi}$ exists, then the homothety $H_{\psi}=mH_{\phi}$ associates to $\psi=i_m\comp\phi$ (recall that $i_m(g)=mgm^{-1}$). For two isogredient iwip endomorphisms $\psi=i_c\comp\phi\comp i_{c^{-1}}=i_{c\phi(c^{-1})}\comp\phi$, we have
$$H_{\psi}=c\phi(c^{-1})H_{\phi}, \quad C_{H_{\psi}}=\fix(H_{\psi})=c(\fix(H_{\phi}))=c(C_{H_{\phi}}).$$
In other words, centers of homotheties of two isogredient iwip endomorphisms are in the same orbit of $F_n$.
\end{lem}

\begin{lem}\label{H free}
	Let $C_H$ be the center of the homothety $H$ associated to an iwip endomorphism $\phi$. Then $\fix\phi\subset \stab(C_H)$. If $C_H$ doesn't exist, then $\fix\phi$ is trivial.
\end{lem}

\begin{proof}
If $w\in \fix\phi$, i.e., $\phi(w)=w$, then by Equation \ref{homothety equa.}, we have
$$H(wC_H)=\phi(w)H(C_H)=wC_H,$$
which implies that $wC_H=C_H$ and hence $w\in\stab(C_H)$. If $C_H$ doesn't exist, we have $\phi\notin \aut(F_n)$, and hence $\fix\phi$ is trivial by \cite[Proposition 3.11]{R}.
\end{proof}

As a corollary, we have

\begin{prop}\label{1 cor}
Let $\phi$ be an iwip endomorphism but $\phi\notin \aut(F_n)$. Then $\fix\phi$ is trivial, and hence $a(\phi)=\#\mathcal A(\phi)$.
Moreover, $\partial \phi$ has at most $2n$ attracting fixed points.
\end{prop}

\begin{proof}
Recall that $a(\phi)$ denotes the number of orbits of $\fix\phi$ acting on $\mathcal A(\phi)$. By Proposition \ref{2 prop}, the stable tree $T_\Phi$ is free, so the stabilizer $\stab(C_H)$ is trivial. Then $\fix\phi$ is trivial according to Lemma \ref{H free}, and hence, every orbit in $\mathcal A(\phi)$ consists of a unique attracting fixed point. Moreover, since $\ind(\phi)\leq n-1$ according to Theorem \ref{ZZ1 thm2.6'}, $\partial\phi$ has at most $2n$ attracting fixed points.
\end{proof}

For an iwip endomorphism $\phi:F_n\rightarrow F_n$  and $\phi\notin \aut(F_n)$, since the stable tree $T_{\Phi}$ is free, we can extend the result in \cite[Lemma 3.5]{GJLL} to iwip endomorphisms without modification:

\begin{lem}\label{inj of j}
Let $\phi:F_n\rightarrow F_n$ be an iwip endomorphism but $\phi\notin \aut(F_n)$. Then there is an $F_n$-equivariant injective map
$$j: \partial T_{\Phi}\rightarrow \partial F_n.$$
If $H$ is a homothety of $T_\Phi$ as in Equation \ref{homothety equa.}, then $\phi(j(\rho))=j(H(\rho))$ for any ray $\rho\in\partial T_{\Phi}$.
\end{lem}

Following \cite{GJLL}, we have the following definition:

If the homothety $H: T_\Phi\to T_\Phi$ sends a component $A$ of $T_{\Phi}\setminus C_H$ to itself, then $A$ contains a unique $H$-invariant ray $\rho$ at the center $C_H$ of $H$, that is, $\rho(\lambda t)=H\rho(t)$ for $t>0$. Such rays are called \emph{eigenrays} of $H$. Note that the closure $\bar{\rho}$ contains $[C_H,a]\cap [C_H,H(a)]$ for every $a\in A$. Moreover,

\begin{prop}\label{3 prop}
	If $X\in F_n$ is an attracting fixed point of an iwip endomorphism $\phi$ of $F_n$ but $\phi\notin \aut(F_n)$, $H$ is the homothety associated to $\phi$ and $C_H$ exists, then there exists a unique eigenray $\rho$ of $H$ such that $X=j(\rho)$. If $C_H$ doesn't exist, then there is a unique attracting
fixed point in $\partial F_n$.
\end{prop}

\begin{proof}
If there exists $C_H$, note that $\stab(C_H)$ is trivial, then the above conclusion follows from \cite[Proposition 4.3]{GJLL} and Lemma \ref{inj of j} directly. If $C_H$ doesn't exist, then $\fix\phi$ is trivial from Lemma \ref{H free}, and hence $a(\phi)=1$ by \cite[Theorem 3.4]{ZZ2}, that is, there is a unique attracting
fixed point in $\partial F_n$.
\end{proof}

Now we have a key proposition as follows.

\begin{prop}\label{key prop}
Let $\Phi$ be an iwip outer endomorphism of $F_n$ but $\Phi\notin \out(F_n)$, and let $H: T_\Phi\to T_\Phi$ be the homothety associated to an iwip endomorphism $\phi\in \Phi$. If $C_H$ exists, then
\begin{enumerate}
\item  there exists an injective map $\tau:\mathcal{A}(\phi)\rightarrow \pi_0(T_{\Phi}\setminus C_H)$, with image
$$\tau(\mathcal{A}(\phi))=\pi_0^H(T_{\Phi}\setminus C_H),$$
where $\pi_0(T_{\Phi}\setminus C_H)$ is the set of components of $T_{\Phi}\setminus C_H$, and $\pi_0^H(T_{\Phi}\setminus C_H)\subset\pi_0(T_{\Phi}\setminus C_H)$ is the set of components fixed by $H$;

  \item $a(\phi)=\#\pi_0^H(T_{\Phi}\setminus C_H)$;

  \item there exists a positive power $\phi^k$ such that $a(\phi^k)=\#\pi_0(T_{\Phi}\setminus C_H)$, and hence
  $$2\ind(\phi^k)=\ind_{\mathrm{geo}}(C_H);$$

  \item the periodic points of $\partial\phi$ are one-to-one corresponding to the components of $T_{\Phi}\setminus C_H$, and hence $\partial\phi$ has at most $2n$ periodic points.
\end{enumerate}
 Otherwise, if $C_H$ doesn't exist, then there is a unique periodic point in $\partial F_n$ and $a(\phi^k)=1$ for all $k\geq 1.$
\end{prop}

\begin{proof}

If $C_H$ doesn't exist, then every homothety $H^k (k\geq 1)$ doesn't have a center (because the expansion factor $>1$). Note that $\fix\phi^k$ is trivial from Lemma \ref{H free}, then by \cite[Theorem 3.4]{ZZ2}, we have $a(\phi^k)=1$ for all $k\geq 1$.

We assume that $C_H$ exists below.

(1) Let $X$ be an attracting fixed point. Then there exists an eigenray $\rho_X$ of $H$ such that $X=j(\rho_X)$ by Proposition \ref{3 prop}. Now we define the map $\tau$ as follows:
let $\tau$ map $X$ into the component $A_X$ of $T_{\Phi}\setminus C_H$, where $A_X$ contains the eigenray $\rho_X$ of $X$. It is clear that  $\tau$ is well-defined since the eigenray is unique.

Below we show that the map $\tau$ is injective. If $X\neq X'$ and $\tau(X)=\tau(X')=A_X$, namely, the overlap of eigenrays $\rho_X$ and $\rho_{X'}$ is a non-degenerate arc, then we get $\rho_X=\rho_{X'}$ since the eigenray is unique in a component fixed by $H$. From Proposition \ref{3 prop}, $$X=j(\rho_X)=j(\rho_{X'})=X',$$
so $\tau$ is injective. From the definition of eigenrays, we get $\tau(\mathcal{A}(\phi))=\pi_0^H(T_{\Phi}\setminus C_H)$.

(2) The assertion follows from the above assertion (1) and Proposition \ref{1 cor} directly.

(3) Note that the homothety $H$ acting on $\pi_0(T_\Phi\setminus C_H)$ is a permutation and $H^i$ associates to $\phi^i$ for any $i>0$. Moreover, by Proposition \ref{2n-2 for endo}, $\pi_0(T_\Phi\setminus C_H)$ is finite. So there is a positive power $H^k$ fixing every element of $\pi_0(T_{\Phi}\setminus C_H)$, and hence
$$a(\phi^k)=\#\pi_0(T_\Phi\setminus C_H)$$
follows from assertion (2). Recall that $\fix(\phi^k)=\stab(C_H)=\{1\}$, by comparing the definitions of $\ind(\phi)$ and $\ind_{\mathrm{geo}}(C_H)$ (Equations \ref{def ind for Out} and \ref{def of geo ind}), we have $2\ind(\phi^k)=\ind_{\mathrm{geo}}(C_H).$

(4) Since $\#\pi_0(T_\Phi\setminus C_{H})\leq 2n$ according to Proposition \ref{2n-2 for endo}, we can find an enough large power $H^k$ such that every component of $T_{\Phi}\setminus C_H$ is fixed by $H^k$. So every component corresponds to a unique eigenray that associates to an attracting fixed point of $\phi^k$, which is a periodic point of $\partial\phi$. Conversely, if $X$ is a periodic point of $\partial\phi$ with period $p$, then $X$ is an attracting fixed point of $\partial\phi^p$. So by assertion (1), there exists a unique component $A_X$ which is fixed by $H^p$.
\end{proof}

\begin{proof}[\textbf{Proof of Theorem \ref{main thm2}}]
Let $\Phi$ be an iwip outer endomorphism of $F_n$. Note that for all $m>0$, the stable tree $T_{\Phi^m}=T_\Phi$ and hence $\ind_{\mathrm{geo}}(T_{\Phi^m})=\ind_{\mathrm{geo}}(T_\Phi)$.

If $\Phi\in \out(F_n)$, the result is from \cite[Section 5A]{CH1} and \cite[Case 5.A]{GJLL}.

If $\Phi\notin \out(F_n)$, by Proposition \ref{2n-2 for endo}, we have
$$\ind_{\mathrm{geo}}(T_{\Phi^m})=\ind_{\mathrm{geo}}(T_{\Phi})=2n-2$$ for all $m>0$.

Now we assume that $\Phi\notin \out(F_n)$ and every branch point $x\in G_\Phi$ is $h$-periodic (i.e. $h^t(x)=x$ for some $t>0$ depending on $x$) in the following. Let $p: T_\Phi \rightarrow G_\Phi$ be the quotient map by the $F_n$-action. Since there are at most $2n -2$ orbits of branch points in $T_\Phi$, we obtain that $G_\Phi$ has at most $2n -2$ branch points. Therefore, there is a uniform power $h^\ell$ ($\ell>0$) fixing all the branch points in $G_\Phi$. Let $Q\in p^{-1}(x)$ be a branch point in $T_\Phi$. Then there exits $c\in F_n$ such that $c^{-1}H^\ell(Q)=Q$, that is, $Q$ is the center of $c^{-1}H^\ell$. Note that $c^{-1}H^\ell$ is associated to the iwip endomorphism $i_{c^{-1}}\comp \phi^\ell\in \Phi^\ell$, and hence, every branch point of $G_\Phi$ is $h$-periodic if and only if every branch point $Q\in T_{\Phi}$ can be a center of a homothety associated to some iwip endomorphism $\varphi\in \Phi^\ell$. So, following from the proof of Item (3) in Proposition \ref{key prop}, there exists a positive power $\varphi^s\in \Phi^k$ ($k=\ell s>0$ can be uniform and independent on $Q$ because $\#\pi_0(T_\Phi\setminus Q)$ is finite) such that
\begin{equation}\label{equa main}
2\ind(\varphi^{s})=\ind_{\mathrm{geo}}(Q)>0.
\end{equation}
Moreover, note that every endomorphism $\psi\in \Phi^{k}$ such that $\psi$ associated the homothety with branch point center has positive index. From Lemma \ref{homothety equa.}, we have one-to-one correspondences:
\begin{eqnarray}
&&\{\mathrm{isogrediency ~class ~in} ~\Phi^{k} ~\mathrm{with ~positive ~index}\}\nonumber\\
&\leftrightarrow& \{\mathrm{orbits ~of ~branch ~points ~in} ~T_{\Phi}\}\nonumber\\
&\leftrightarrow& \{\mathrm{orbits ~of ~branch~centers ~of ~homotheties ~associated ~to ~endomorphisms ~in} ~\Phi^{k}\}\nonumber.
\end{eqnarray}
Now combining the above correspondences, Equation \ref{equa main} and the definitions of $\ind(\Phi)$ and $\ind_{\mathrm{geo}}(T_\Phi)$ (Equations \ref{def ind for Out} and \ref{def of geo ind}), we have
 $$2\ind(\Phi^{k})=\ind_{\mathrm{geo}}(T_\Phi)=2n-2.$$

If $\Phi\notin \out(F_n)$ and some branch point $x\in G_\Phi$ is not $h$-periodic, then every branch point $Q\in p^{-1}(x)$ is not a center of any homothety induced by a positive power of $h$, and there is an iwip endomorphism $\psi\in \Phi^m$ associated to the homothety $H_\psi$ without center for all $m>0$ (see Example \ref{example referee}). Following from Lemma \ref{H free} and Proposition \ref{key prop}, we have
$\ind(\psi)=0+1/2-1<0$ and $\ind_{\mathrm{geo}}(Q)>0$.
Therefore, by the definition of geometric index and Proposition \ref{2n-2 for endo}, we have
$$ 2\ind(\Phi^m)<\ind_{\mathrm{geo}}(T_\Phi)=2n-2$$
for all $m>0.$ The proof is completed.
 \end{proof}

\begin{rem}
Note that it is possible for a center to be bivalent (i.e. not a branch point), but such a center corresponds to an isogrediency class with index $=0$. So, in the one-to-one correspondences in the above proof of Theorem \ref{main thm2}, the third part should be orbits of ``branch centers".
\end{rem}

\section{Proof of Theorem \ref{main thm3} and two examples}\label{sect last}

\begin{thmbis}{main thm3}
Suppose $X$ is a connected finite graph but not a tree, and $f: X\rightarrow X$ is a selfmap. Let $f_\pi$ induce an iwip outer endomorphism $\Phi$ of the fundamental group $\pi_1(X)=F_n$. Then there is an integer $k\geq 1$ such that
$$\sum_{\ind(f^k, \F)+\chr(f^k, \F)<0} \{\ind(f^k, \F)+\chr(f^k, \F)\}= 2\chi(X),$$
if and only if
\begin{enumerate}
	\item $\Phi\in \out(F_n)$ and the stable tree $T_\Phi$ is geometric; or
	\item $\Phi\notin \out(F_n)$ and every branch point in $G_\Phi$ is $h$-periodic.
\end{enumerate}
\end{thmbis}

\begin{proof}
Let $\pi_1(X)=F_n$, the free group of rank $n$. Then the Euler characteristic $\chi(X)=1-n\leq 0$. Recall that an iwip endomorphism is injective, then by Theorem \ref{2 thm}, to prove Theorem \ref{main thm3}, is equivalent to prove: there is a positive power $\Phi^k$ such that $$2\ind(\Phi^k)=2n-2,$$
if and only if (1) $\Phi$ is an iwip outer automorphism and the stable tree $T_\Phi$ is geometric, or (2) $\Phi\notin\out(F_n)$ is an iwip outer
endomorphism and every branch point in $G_\Phi$ is $h$-periodic,  which holds according to Theorem \ref{main thm2}. So the proof is completed.
\end{proof}

Finally, let us give two examples that do not satisfy Theorem \ref{main thm3}.

\begin{exam}[Jiang, \cite{J09}]
Let $f: (R_2, \ast) \to(R_2, \ast)$ be a $\pi_1$-injective selfmap of the graph $R_2$ with one vertex $\ast$ and two edges $a, b$, such that $f(a)=a$ and
$f(b)=\bar bab$.

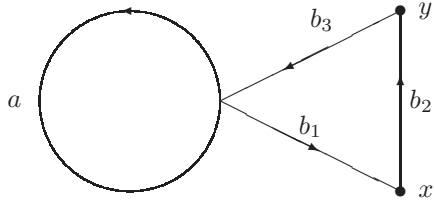
\begin{center}
\setlength{\unitlength}{1.2mm}
\begin{picture}(42,25)(0,-10)
\qbezier(20,0)(20.,1.00504)(19.8,1.98998)
\qbezier(19.8,1.98998)(19.6,2.97492)(19.208,3.90037)
\qbezier(19.208,3.90037)(18.816,4.82581)(18.2477,5.65473)
\qbezier(18.2477,5.65473)(17.6793,6.48366)(16.9574,7.18291)
\qbezier(16.9574,7.18291)(16.2355,7.88216)(15.3889,8.42377)
\qbezier(15.3889,8.42377)(14.5423,8.96537)(13.6048,9.32767)
\qbezier(13.6048,9.32767)(12.6673,9.68997)(11.6765,9.85846)
\qbezier(11.6765,9.85846)(10.6857,10.027)(9.68116,9.99492)
\qbezier(9.68116,9.99492)(8.67663,9.96287)(7.69857,9.73157)
\qbezier(7.69857,9.73157)(6.7205,9.50026)(5.80803,9.07895)
\qbezier(5.80803,9.07895)(4.89556,8.65764)(4.08517,8.06318)
\qbezier(4.08517,8.06318)(3.27479,7.46871)(2.59891,6.72487)
\qbezier(2.59891,6.72487)(1.92303,5.98103)(1.4087,5.11757)
\qbezier(1.4087,5.11757)(0.89436,4.25411)(0.562137,3.30556)
\qbezier(0.562137,3.30556)(0.229914,2.35702)(0.0930946,1.36133)
\qbezier(0.0930946,1.36133)(-0.043725,0.365647)(0.0203316,-0.637352)
\qbezier(0.0203316,-0.637352)(0.0843881,-1.64035)(0.346759,-2.61054)
\qbezier(0.346759,-2.61054)(0.609129,-3.58073)(1.05932,-4.47931)
\qbezier(1.05932,-4.47931)(1.50951,-5.37789)(2.12951,-6.1689)
\qbezier(2.12951,-6.1689)(2.74951,-6.95992)(3.51452,-7.61174)
\qbezier(3.51452,-7.61174)(4.27953,-8.26356)(5.15896,-8.7501)
\qbezier(5.15896,-8.7501)(6.03838,-9.23665)(6.99703,-9.53846)
\qbezier(6.99703,-9.53846)(7.95569,-9.84027)(8.95523,-9.94527)
\qbezier(8.95523,-9.94527)(9.95477,-10.0503)(10.9552,-9.95427)
\qbezier(10.9552,-9.95427)(11.9557,-9.85827)(12.917,-9.5651)
\qbezier(12.917,-9.5651)(13.8783,-9.27193)(14.7621,-8.79332)
\qbezier(14.7621,-8.79332)(15.6459,-8.31471)(16.4167,-7.6698)
\qbezier(16.4167,-7.6698)(17.1876,-7.0249)(17.8147,-6.23949)
\qbezier(17.8147,-6.23949)(18.4417,-5.45408)(18.9,-4.5596)
\qbezier(18.9,-4.5596)(19.3583,-3.66511)(19.6294,-2.69732)
\qbezier(19.6294,-2.69732)(19.9004,-1.72953)(19.9735,-0.727146)
\qbezier(19.9735,-0.727146)(20.,-0.364055)(20,0)
\put(20,00){\line(2,-1){20}}
\put(20,00){\line(2, 1){20}}
\put(40,-10){\line(0, 1){20}}

\put(11,10){\vector(-1,0){2}}
\put(26,-3){\vector(2,-1){5}}
\put(40,-2){\vector(0, 1){5}}
\put(32, 6){\vector(-2,-1){5}}

\put(-2,0){\makebox(0,0)[rc]{$a$}}
\put(41,0){\makebox(0,0)[lc]{$b_2$}}
\put(30,9){\makebox(0,0)[lc]{$b_3$}}
\put(30,-4){\makebox(0,0)[cb]{$b_1$}}

\put(40,-10){\makebox(0,0)[cc]{$\bullet$}}
\put(42,-10){\makebox(0,0)[lc]{$x$}}
\put(40,10){\makebox(0,0)[cc]{$\bullet$}}
\put(42,10){\makebox(0,0)[lc]{$y$}}

\end{picture}
\end{center}

Consider that $a$ gets initially expanded along itself on one tip and shrinks on the other.
By inserting two new vertices $x, y$ in the interior of $b$, we break $b$ into three edges $b=b_1b_2b_3$ such that
$$f(a)=a, ~~f(b_1)=\bar b, ~~f(b_2)=a, ~~f(b_3)=b.$$
If we write the path classes of the edges $a$ and  $b$ still by $a$ and $b$ respectively, then the fundamental group $\pi_1(R_2, \ast)=\langle a, b\rangle\cong F_2$, with $f_{\ast}: \pi_1(R_2, \ast)\to \pi_1(R_2, \ast)$ defined by $f_{\ast}(a)=a$ and
$f_{\ast}(b)=b^{-1}ab$.

Note that $f^k_\ast$ fixes the free factor $\langle a \rangle$ of $\pi_1(R_2, \ast)=\langle a, b\rangle$ for every $k\geq 1$. It implies that the injective endomorphism $f_\ast$ is not iwip with $\fix(f^k_{\ast})=\langle a\rangle\cong \Z$.
Moreover, the fixed point class $\F=\{\ast\}$ with
$$\ind(f^k, \F)=-1,  \quad \chr(f^k, \F)=1-\rk\fix(f^k_{\ast})=0$$
is the unique fixed point class contributes to  the left side of inequality (B) in Theorem \ref{JWZ main theorem}.
Therefore
$$\ind(f^k, \F)+\chr(f^k, \F)=-1+0> -2=2\chi(R_2)$$
for every $k\geq 1$.
\end{exam}

\begin{exam}\label{example referee}
	Let $F_2=\langle a,b \rangle$ and $\phi:F_2\rightarrow F_2$ be an endomorphism given by $\phi(a)=ab$ and $\phi(b)=ab^{-1}a.$ The outer endomorphism $\Phi$ represented by $\phi$ should be an iwip. Let the graph $G_\Phi$ be the quotient of the stable tree $T_\Phi$ by the $F_n$-action. Denote by $h:G_\Phi\rightarrow G_\Phi$ the map induced by the homothety $H$. Then we have $h:G_\Phi\rightarrow G_\Phi$ a map on the theta graph defined as follows:
	\begin{enumerate}
		\item Let $V(G_\Phi)=\{v_1,v_2\}$ and $E(G_\Phi)=\{e_1,e_2,e_3\}.$ Orient each edge to start at $v_1$ and end at $v_2$. Identify $a\in F_2$ with the loop $e_2\overline{e}_1$ and $b\in F_2$ with the loop $e_2\overline{e}_3;$
		\item The map $h:G_\Phi\rightarrow G_\Phi$ induced by $H:T_\Phi\rightarrow T_\Phi$ will map the edges $e_1,e_2,e_3$ to the loops $e_3\overline{e}_2$,  $e_2\overline{e}_1$, $e_1\overline{e}_3$ respectively.
	\end{enumerate}

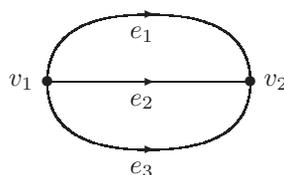
\begin{center}
	\setlength{\unitlength}{0.9mm}
	\begin{picture}(70,25)(-35,-12.5)
		\put(-14,0){\line(1,0){30}}
		
		\qbezier(16,0)(16,10)(1,10)
		\qbezier(-14,0)(-14,10)(1,10)
		\qbezier(16,0)(16,-10)(1,-10)
		\qbezier(-14,0)(-14,-10)(1,-10)

		\put(0,0){\vector(1,0){2}}
		\put(0,10){\vector(1,0){2}}
		\put(0,-10){\vector(1,0){2}}
		
		\put(0,6){\makebox(0,0)[cb]{$e_1$}}
		\put(0,-4){\makebox(0,0)[cb]{$e_2$}}
		\put(0,-14){\makebox(0,0)[cb]{$e_3$}}
		
		\put(-14,0){\makebox(0,0)[cc]{$\bullet$}}
		\put(-16,0){\makebox(0,0)[rc]{$v_1$}}
		\put(16,0){\makebox(0,0)[cc]{$\bullet$}}
		\put(18,0){\makebox(0,0)[lc]{$v_2$}}
		
	\end{picture}
\end{center}

Note that the vertex $v_2$ is not $h$-periodic, and $h$ has two fixed points $v_1$ and $v'$ (where $v'$ is the fixed point in the interior of the edge $e_3$ since $e_3$ maps over itself with reverse orientation). So $\phi$ has branch center $v_1$ and exactly one other isogrediency class $\phi'\in\Phi$ has center $v'$. Since $v'$ is bivalent, the isogrediency class $\phi'$ doesn't contribute to the index of $\Phi$. In general, more isogrediency classes in $\Phi^k$ can have centers (because $h^k$ can have more fixed points). But all these new centers are bivalent, and the corresponding isogrediency classes in $\Phi^k$ have index $0$. So they also don't contribute to the index of $\Phi^k$, and hence
$$2\ind(\Phi^k)=2\ind(\phi^k)\leq 1$$
for all $k\geq1$. In conclusion, if we remove the condition ``\emph{every branch point in $G_\Phi$ is $h$-periodic}" from Theorem \ref{main thm2}, then the equality $2\ind(\Phi^k)=2n-2$ will not hold.
\end{exam}

\noindent\textbf{Acknowledgements.}  The authors would like to thank the anonymous referee very much for his/her valuable
and detailed comments helped to greatly improve our earlier draft, especially for pointing out an error in the proof of the main theorems, and providing detailed corrections and Example \ref{example referee}. 

\end{document}